\documentclass[11pt, oneside]{amsart}
\usepackage[margin=1.2in,footskip=0.25in]{geometry}
\usepackage{amsmath, amssymb, latexsym}
\usepackage[T1]{fontenc}

\usepackage{xcolor}
\usepackage{graphicx}
\usepackage{amsfonts}
\usepackage{amsthm}
\usepackage{hyperref}
\newtheorem{theorem}{Theorem}
\newtheorem{definition}[theorem]{Definition}
\newtheorem{corollary}[theorem]{Corollary}
\newtheorem{proposition}[theorem]{Proposition}
\newtheorem{lemma}[theorem]{Lemma}

\newtheorem{example}[theorem]{Example}
\usepackage{enumitem}
\usepackage{comment}
\usepackage{multirow}

\def\Fq{\mathbb{F}_q}
\def\Fp{\mathbb{F}_p}

\usepackage{listings}

\title{On $\mathbb{F}_q$-primitive points on hypersurfaces}


\begin{document}

\author[J. A. Oliveira]{Jos\'e Alves Oliveira}
\address{Departamento de Matem\'{a}tica,
	Universidade Federal de Lavras,
	UFLA,
	Lavras, MG 37200-900 ,
	Brazil.}
\curraddr{}
\email{jose{\string_}oliveira@ufla.br}
\thanks{}

\author[Marcelo O. Veloso]{Marcelo Veloso} 
\address{Departamento de F\'isica e Matem\'{a}tica,
	Universidade Federal de São João del-Rei,
	UFSJ, Alto Paraopeba,
	Ouro Branco, MG 36415-000 ,
	Brazil.}
\curraddr{}
\email{veloso@ufsj.edu.br}
\thanks{}

	\maketitle
	\begin{abstract}
		\noindent In this paper, we estimate the number of $\mathbb{F}_q$-primitive points on the affine hypersurface defined by the equation $f(x_1,\ldots,x_s)=0$, where $f\in\mathbb{F}_q[x_1,\dots,x_s]$ is an appropriate polynomial. In particular, we provide existence results for the case when $f$ is Dwork-regular and when $f$ is of Fermat type. Additionally, we present a proof for a recently posed conjecture. Finally, in the case where $q$ is a Fermat prime, we provide an explicit formula for the number of $\mathbb{F}_q$-primitive points on hyperplanes.
	\end{abstract}
	
	

\section{Introduction}
Let $\Fq$ be a finite field with $q$ elements, where $q$ is a power of the characteristic $p$. The multiplicative group $\Fq^*$ is cyclic, and its generators are called \textit{primitive} elements of $\Fq$. These elements are important in the theory of finite fields, primarily due to their applications. This interest has motivated research into problems involving primitive elements. The most notable example is the well-known Primitive Normal Basis Theorem, which was first proven by Lenstra and Schoof in~\cite{lenstra} using a computer, and later by Cohen and Huczynska in~\cite{cohen2003} without the assistance of a computer. In a slightly different vein, the description of finite fields that contain a pair $(\alpha, f(\alpha))$ of primitive elements has emerged as a recurring theme in this area~\cite{Booker,Cohen85,Cohen21,Wang}. In a notable paper \cite{cohen15}, the authors provide a positive answer to a conjecture posed by Cohen and Mullen, which guarantees the existence of a pair $(\alpha,f(\alpha))$ of primitive elements in the case where $f$ is a general linear polynomial and $q>61$. Overall, results in this domain typically rely on character sums to establish cases for $q > q_0$, complemented by computational methods for the finitely many scenarios where $q < q_0$, where $q_0$ is a fixed constant.
 
 As noted by Cohen, Kapetanakis and Reis~\cite{cohenreis}, the number of pairs $(\alpha, f(\alpha))$ of primitive elements corresponds to the number of $\Fq$-rational points on the affine curve 
 $$\mathcal C: y=f(x),$$
 whose coordinates are primitive. In this context, they introduced the concept of $\Fq$-primitive point on a curve, which refers to $\Fq$-rational points where all coordinates are a primitive element of $\Fq$. They also defined the notion of $(r,d)$-free elements in cyclic groups, a convenient framework for representing elements expressible in specific forms. With this machinery in hand, the authors extended existing results in the literature by proving existence theorems for affine curves described by equations of the form $y^n=f(x)$. 
 Extending the definition of $\mathbb{F}_q$-primitive points to affine hypersurfaces in $s$ variables is a natural progression of this problem. Indeed, the case $s=3$ was recently studied by Takshak, Kapetanakis, and Sharma~\cite{Takshak2025}, where the authors estimate the number of $\mathbb{F}_q$-primitive points on affine surfaces of the type $y^n = f(x,y)$, with $g$ satisfying certain conditions.

In this paper, we address the number of $\mathbb{F}_q$-rational points on affine hypersurfaces whose coordinates are primitive, which we refer to as $\mathbb{F}_q$-primitive points, thereby extending the aforementioned concept. Along the paper, we let $P_q(f)=P_q(f(x_1,\dots,x_s))$ denote the number of $\mathbb{F}_q$-primitive points on the affine hypersurface
\begin{equation}\label{hypersurface}
    \mathcal H:f(x_1,\dots,x_s)=0.
\end{equation}
Since there exist $\varphi(q-1)$ primitive elements in $\Fq$, where $\varphi$ denotes the Euler totient function, it is heuristically expected that a proportion of $\tfrac{\varphi(q-1)}{q-1}$ of the non-zero values of $x_i$ in ~\eqref{hypersurface} will be primitive. Consequently, it is expected that
$$P_q(f)\approx \left(\tfrac{\varphi(q-1)}{q-1}\right)^s |\mathcal H(\Fq)^*|,$$
where $\mathcal H(\Fq)^*$ represents the set of $\Fq$-rational points on $\mathcal H$ whose coordinates are non-zero. These heuristics have been demonstrated to hold true in instances where $\mathcal H$ is a superelliptic curve given by $y^n=f(x)$ (see Theorem 16 in~\cite{cohenreis} and Lemma~\ref{lemad}) and where $\mathcal H$ is a superelliptic surface given by $z^n=f(x,y)$ (see the proof of Theorem 3.2 in~\cite{Takshak2025} and Lemma~\ref{lemad}). A natural generalization of these specific cases arises when $\mathcal H$ is a affine superelliptic hypersurface of the form 
$$ y^n=f(x_1,\ldots,x_s).$$
As noted in Remark 3.4 of \cite{Takshak2025}, their approach can also be applied in this more general setting, assuming that $f\in\Fq(x_1,\ldots,x_s)$ is a \textit{primarily non-exceptional} rational function, a minor restriction on the shape of $f$ that avoids some exceptional or awkward cases. Indeed, by applying the techniques used in the proof of Theorem 3.2 in~\cite{Takshak2025} together with Lemma~\ref{lemad}, one obtains the following result.
\begin{theorem}\label{cohenapproach}
    Let $f\in\Fq(x_1,\ldots,x_s)$ be a primarily non-exceptional rational function of degree sum $d$. Then
    $$\left|P_q\big(y^n-f(x_1,\ldots,x_s)\big)- \tfrac{\varphi(q-1)^{s+1}}{q}\right|\le ndW\big(\tfrac{q-1}{n}\big) W\big(q-1\big)\varphi(q-1)^{s+1}(q-1)^{-2} q^{\frac{1}{2}},$$
where $W(m)$ denotes the number of squarefree divisors of $m$ and $\varphi(m)$ denotes the Euler's totient function.
\end{theorem}
While this result covers a broad class of superelliptic hypersurfaces, the presence of the term $\varphi(q-1)^{s+1}$ weakens the bound. This is mainly due to the fact that Weil's bound was applied to only a single variable. This discussion naturally leads to two important questions: can one obtain a tighter bound for a specific family of hypersurfaces? Can a similar result be established in cases where $\mathcal H$ is not a superelliptic hypersurface? The primary aim of this paper is to estimate the number of $\Fq$-primitive points on different classes of hypersurface families by employing a novel approach. Among other matters, we compute $P_q(f)$ in terms of character sums and then provide an estimative for $P_q(f)$ in the case where $f$ is a Dwork-regular polynomial or a Fermat-type polynomial. This estimate allow us to prove that $P_q(f)>0$ for sufficiently large $q$. As a consequence of our results, we prove a conjecture recently posed in \cite{Takshak2025}. Furthermore, we provide an explicit formula for $P_q(f)$ in the case where $f(x_1,\ldots,x_s)=a_1 x_1+\ldots a_s x_s-b$ and $q$ is a Fermat prime.

\section{Main results}

To present our main results, we first introduce a class of polynomials. A polynomial $f(x_1,\dots,x_s)\in\Fq(x_1,\dots,x_s)$ is called a \textit{Deligne
 polynomial} if its degree $d$ is prime to $p$, its highest degree term, $f_d$, is a nonzero homogeneous form of degree $d$ in $s$ variables and, for $s \geq 2$, the vanishing of $f_d$ defines a smooth hypersurface in the projective space $\mathbb{P}^{s-1}$. A Deligne polynomial $f = f(x_1,...,x_s)$ is \textit{Dwork-regular} if, for every proper subset $I\subset \{1,2,...,s\}$, the polynomial obtained by setting $x_i \mapsto 0$ for $i\in I$ is a also a Deligne polynomial of the same degree $d$ in the remaining variables. For a Dwork-regular polynomial $f$, we are able to estimate $P_q(f)$ and establish the following lower bound.

\begin{theorem}\label{teo1} If $f\in\Fq[x_1,\ldots,x_s]$ of prime-to-$p$ degree $d$ is Dwork-regular, then
    $$\left|P_q(f)- \tfrac{\varphi(q-1)^s}{q}\right| \le
(d q^{\frac{1}{2}}+1)^s W(q-1)^s,$$
where $W(n)$ denotes the number of squarefree divisors of $n$ and $\varphi(n)$ denotes the Euler's totient function.
\end{theorem}

By using Lemmas~\ref{phibound} and \ref{boundW} on the numbers $\varphi(q-1)$ and $W(n)$, Theorem 1 implies that
\[
\begin{aligned}
P_q(f)>   \frac{(q-1)^{s}}{q}\left(\frac{2ln(ln(q-1))}{2e^{\gamma}[ln(ln(q-1))]^2+5}\right)^s
 -(dq^{\frac{1}{2}}+1)^{s}q^{\frac{0.96s}{ln(ln(q))}},
\end{aligned}
\]
which means that, for every degree $d$ polynomial $f$ that satisfies the conditions, we have $P_q(f)>0$ if $q$ is sufficiently large. The conditions that $f$ is of degree $d$ (prime to $p$) and Dwork-regular are fundamental in the proof of the theorem, as it relies on a result concerning mixed character sums by Katz~\cite{katz}. We strongly believe that such a bound holds for a more general family of polynomials $f$, however, removing this condition appears to be challenging, as discussed throughout Katz's paper. It is worth mentioning that any polynomial of the form $\sum_{i=1}^s x_i^d+$\textit{(terms of lower degree)} is Dwork-regular, which means that this gives a large family of polynomials.

Note that the polynomial $f(x_1,\dots,x_s)= a_1x_1^{d}+a_2x_2^{d}+\cdots+a_sx_s^{d}-b$ is Dwork-regular, which means that Theorem~\ref{teo1} provides the existence of 
$\Fq$-primitive points on the Fermat hypersurface
$$a_1x_1^{d}+a_2x_2^{d}+\cdots+a_sx_s^{d}=b.$$
The number of $\Fq$-rational points on this type of Fermat hypersurface was studied by Wolfmann~\cite{wolfmann} and was recently extended to the case of distinct exponents $d_1, \dots, d_s$ in \cite{oliveira1,oliveira2}. In the proof of the following result, we employ Jacobi sums to derive an estimate for the number of $\Fq$-primitive points on Fermat hypersurfaces.

\begin{theorem}
\label{fermat-cota}
Let $f(x_1,\dots,x_s)= a_1x_1^{d_1}+a_2x_2^{d_2}+\cdots+a_sx_s^{d_s}-b\in\Fq[x_1,\ldots,x_s]$, where $s\geq 2$, and $d_i\mid (q-1)$, for $i=1,\ldots,s$. Then
$$\big|P_q(f)- \tfrac{\varphi(q-1)^s}{q}\big|\le\left(\tfrac{\varphi(q-1)}{q-1}\right)^s\tfrac{1}{\sqrt{q}}\left[1+\bigg(d_1W\big(\tfrac{q-1}{d_1}\big)-1\bigg)q^{\frac{1}{2}}\right]\dots\left[1+\bigg(d_sW\big(\tfrac{q-1}{d_s}\big)-1\bigg)q^{\frac{1}{2}}\right]+\delta_b,$$
where $\delta_b=(\varphi(q-1)/(q-1))^s$ if $b=0$ and $\delta_b=0$ if $b\neq 0$.
\end{theorem}

By applying a sieving version of Theorem~\ref{fermat-cota}, we prove a conjecture recently stated by  Takshaka, Kapetanakis and Sharma \cite{Takshak2025}.

\begin{theorem}\cite[Conjecture  $4.3$]{Takshak2025}\label{conj}
 Assume that $q$ is odd. There exists a $\Fq$-primitive point on the sphere $x^2+y^2+x^2=1$ if and only if   
$q\neq 3,5,9,13,25.$
\end{theorem}

In the case where $q$ is a Fermat prime, we can compute explicitly the number of $\Fq$-primitive points on affine hyperplanes. In order to state the result, we let $\chi_2$ be the quadratic character of $\mathbb{F}_q^*$, $\nu(0)=q-1$ and $\nu(b)=-1$ for $b\in\Fq^*$.
\begin{theorem}\label{numexato}
\label{plinear}
Let $f(x_1,\ldots,x_s)=a_1 x_1+\cdots+a_s x_s-b\in\Fq[x_1,\ldots,x_s]$, where $q>3$ is a Fermat prime. Then
	\[
P_q(f(\mathbf x))=\frac{\varphi(q-1)^s}{q}+\frac{1}{q}\left(\frac{(-1)^s\tau_0 \prod_{i=1}^s\big(\sqrt{q}\chi_2(a_i)+1\big)+\tau_1\prod_{i=1}^s\big(\sqrt{q}\chi_2(a_i)-1\big)}{2^{s+1}}\right),\]
where $\tau_0=\nu(b)+\sqrt{q}\chi_2(b)$ and $\tau_1=\nu(b)-\sqrt{q}\chi_2(b)$.
\end{theorem}
It is worth noting that the existence of infinitely many Fermat primes is still an open question. Nevertheless, we can increase the dimension $s$ of the hyperplane as much as desired, allowing us to obtain an infinite family of hyperplanes. As a direct consequence of Theorem~\ref{numexato}, we have the following result.
\begin{corollary}\label{corexato}
\label{plinear}
Let $f(x_1,\ldots,x_s)=x_1+x_2+\cdots+x_s$ be a linear form over $\mathbb{F}_q$, where $q>3$ is a Fermat prime. Then 
\[P_q(x_1+\cdots+x_s)=\frac{\varphi(q-1)^s}{q}+\frac{q-1}{q}\left(\frac{(\sqrt{q}-1)^s+(-1)^s(\sqrt{q}+1)^s}{2^{s+1}}\right)\]
\end{corollary}

We highlight that the case $q=3$ can also be derived in Theorem~\ref{numexato} and Corollary~\ref{corexato} by applying a suitable modification to the proof of Theorem~\ref{quadratic}. To conclude this section, we present an illustrative example.

\begin{example}
    Consider the equation $x+y+z=0$ on $\mathbb{F}_q^3$, where $q>3$ is a Fermat prime. In this case, Corollary~\ref{corexato} reads
$$P_q(x_1+x_2+x_3)=\frac{\varphi(q-1)^3}{q}+\frac{q-1}{q}\left(\frac{(\sqrt{q}-1)^3-(\sqrt{q}+1)^3}{16}\right)=\frac{q^2-6q+5}{8}.$$
\end{example}


\section{Preliminaries}

In this section we provide some well-known definitions and results that will be important in the proof of the main result. Throughout the text, $\varphi(n)$ denotes the Euler totient function, that is, $\varphi(n)$ is the number of positive integers up to $n$  that are relatively prime to $n$. The following lower bound on $\varphi(n)$, derived from \cite{RosserShoenfel}, will be crucial in our proofs. 
\begin{lemma}\cite[Theorem $15$]{RosserShoenfel}\label{phibound}
Let $\gamma$ be the Euler-Mascheroni constant. For an integer $n\ge 2$, it follows that 
 \begin{equation}
 \label{boundeuler}
 \frac{n}{\varphi(n)} < e^{\gamma}ln(ln(n))+\frac{5}{2ln(ln(n))}    
 \end{equation}
 with one exception for $n = 223092870 = 2 \cdot 3 \cdot 7 \cdot 11 \cdot 13 \cdot 17 \cdot 19 \cdot 23$.
\end{lemma}

By taking the reciprocal of both sides of Equation \ref{boundeuler} we establish a lower bound on the ratio of  $\varphi(n)$ to $n$:

\begin{equation}
    \frac{2ln(ln(n))}{2e^{\gamma}[ln(ln(n))]^2+5} < \frac{\varphi(n)}{n}.
\end{equation}




The Möbius function, denoted by $\mu(n)$, is defined as
\[
\mu(n) = 
  \begin{cases}
      (-1)^k, & \mbox{ if } n \mbox{ is the product of k distinct primes} \\
      \quad 0, & \mbox{ otherwise}.  
  \end{cases}
\]
It is well known (see, for instance, Proposition 2.2.4. of \cite{Ireland}) that
\[
n = \sum_{d \mid n} \varphi(d). 
\]
Using the Möbius inversion formula, we obtain:
\begin{lemma}\label{mobiusEuler} We have that
\[
\varphi(n) = \sum_{d \mid n} \mu(d) \cdot \frac{n}{d}.    
\]
\end{lemma}

\begin{lemma}[Lema 3.5, \cite{ReisRibas}]
\label{boundW}
If $W(t)$ is the number of squarefree divisors of $t$, then 
\[
W(t-1) <t^{\frac{0.96}{ln(ln(t))}},
\]
for all $t\geq 3$.
\end{lemma}


Let $\psi$ be the canonical additive character from $\Fq$ to $\Fp$. For $r$ a divisor of $q-1$, let $\chi_r$ be a multiplicative character of order $r$ of $\Fq^*$. As usual, we extend $\chi_r$ to $\Fq$ by setting $\chi_r(0)=0$. 

\begin{lemma}\cite[Theorem $5.4$]{lidl1997finite}\label{item123} We have that
    $$\sum_{x\in\Fq}\psi(x)=0.$$
\end{lemma}

The following result is a direct consequence of Theorem 5.8 of \cite{lidl1997finite}.

\begin{lemma}\label{quachar} Let $q$ be odd and $\chi_2$ be the quadratic character of $\Fq$. Then
    $$\chi_2(-1)=(-1)^{\frac{q-1}{2}}.$$
\end{lemma}

\begin{lemma}[{{\cite[Equation $5.4$, p. 189]{lidl1997finite}}}]\label{dpower}
Let $d$ be a divisor of $q-1$. If $c \in\Fq$, then
\[\sum_{j=0}^{d-1} \chi_d^j(c) = \begin{cases}
    1, & \text{ if } c=0\\
	d, & \text{ if } c \text{ is a } d \text{-power in } \Fq^*\\
	0,& \text{ otherwise.}
\end{cases}\]
\end{lemma}

\begin{definition}
	Let $\lambda_1,\dots,\lambda_s$ be multiplicative characters of $\Fq^*$ and $b\in\Fq$. The \textit{Jacobi sum} of $\lambda_1,\ldots, \lambda_s$ is defined as 
		$$J_b(\lambda_1,\ldots,\lambda_s)=\sum_{\substack{b_1+\cdots+ b_s=b\\ (b_1,\ldots,b_s)\in\Fq^s}}\lambda_1(b_1)\cdots\lambda_s(b_s).$$
\end{definition}

One can verify that 
$$J_b(\lambda_1,\ldots,\lambda_s)=\lambda_1(b)\cdots\lambda_s(b)J_1(\lambda_1,\ldots,\lambda_s)$$
for all $b\in\Fq^*$, fact that will be extensively used in the paper. Throughout the paper, we set $J(\lambda_1,\ldots,\lambda_s)=J_1(\lambda_1,\ldots,\lambda_s)$.

\begin{proposition}\cite[Theorem $5.19$]{lidl1997finite}\label{item1519} Let $\lambda_1,\ldots,\lambda_s$ be multiplicative characters of $\Fq^*$, then 
	$$J(\lambda_1,\ldots,\lambda_s)=\begin{cases}
	q^{s-1},&\text{ if } \lambda_1,\dots,\lambda_s\text{ are all trivial}\\
	0,&\text{ if some, but not all, of }\lambda_1,\dots,\lambda_s\text{ are trivial.}\\
	\end{cases}$$
\end{proposition}

\begin{proposition}\cite[Theorem $5.22$]{lidl1997finite}\label{item519} Let $\lambda_1,\ldots,\lambda_s$ be nontrivial multiplicative characters of $\Fq^*$, then 
	$$|J(\lambda_1,\ldots,\lambda_s)|=\begin{cases}
	q^{\frac{s-1}{2}},&\text{ if } \lambda_1\cdots\lambda_s\text{ is nontrivial}\\
	q^{\frac{s-2}{2}},&\text{ if }\lambda_1\cdots\lambda_s\text{ is trivial.}\\
	\end{cases}$$
\end{proposition}

The following result is a particular case of Theorem 1.1 of Katz~\cite{katz}.
\begin{theorem}\label{katz}
    Let $\lambda_1,\dots,\lambda_s$ be multiplicative characters of $\Fq^*$. If $f\in\Fq[x_1,\dots,x_s]$ is a prime-to-$p$ degree $d$ is Dwork-regular polynomial, then
    $$\left|\sum_{(x_1,\dots,x_s)\in \Fq^s}\psi(f(x_1,\dots,x_s))\lambda_1(x_1)\cdots\lambda_s(x_s)\right|\le d^s q^{\frac{s}{2}}.$$
\end{theorem}


\section{The number of $\mathbb{F}_q$-primitive points}\label{secbasic}

This section aims to estimate the number of $\Fq$-primitive points on the affine hypersurface
\[
\mathcal H:f(x_1,\dots,x_s)=0,
\]
in the case when $f(x_1,\dots,x_s)\in\Fq(x_1,\dots,x_s)$ satisfies a regularity condition.

 The following notation is used throughout this text.

 \begin{itemize}
    \item $\mathbf{x}=(x_1,\dots,x_s)\in\Fq^s$
    \item $I\subset\{1,\ldots,s\}$ and $I^{\mathsf{c}}=\{1,\ldots,s\}\setminus I$
    \item $\mathbf r=(r_1,\dots,r_s)$, where $r_1,\dots,r_s$ are positive integers
    \item $\mathbf{x}^{\mathbf r}=(x_1^{r_1},\dots,x_s^{r_s})$
    \item $\mathbf r\mid(q-1)$ means that each $r_i$ divides $q-1$.
     \item $N(f(\mathbf{x}))$ is the number of $\mathbb{F}_q$-rational points on $\mathcal H$
     \item $N^*(f(\mathbf{x}))$ is the number of $\mathbb{F}_q$-rational points on $\mathcal H$  whose coordinates are nonzero
     \item $N_I(f(\mathbf{x}))$ is the number of $\mathbb{F}_q$-rational points on $\mathcal H$ such that $x_i=0$ for each $i\in I^{\mathsf{c}}$
     \item for integers $r$ and $d$, let $(r,d)=\gcd(r,d)$
 \end{itemize}\vskip0.2cm



It is direct to see that the number $N^*(f(\mathbf{x}))$ is closely related to the number $N_I(f(\mathbf{x}))$. By using the Inclusion-Exclusion Principle, one can prove that:

\begin{lemma}
\label{N*=NI} We have that
\[
N^*(f(\mathbf x))= 
\displaystyle\sum_{I\subset\{1,\ldots,s\}} (-1)^{s-|I|}
 N_I(f(\mathbf x)).
\]
\end{lemma}






Now, we are able to present a expression for the number of $\Fq$-primitive points on the hypersurface $\mathcal H:f(x_1,\dots,x_s)=0$.

\begin{proposition}
\label{pf} The number of $\Fq$-primitive points on $\mathcal H$ is given by
\[
P_q(f(\mathbf x))= 
\displaystyle\sum_{\mathbf r\mid (q-1)}
\frac{\mu(r_1)\cdots\mu(r_s)}{r_1\cdots r_s} N^*(f(\mathbf x^{\mathbf r})) .
\]
\end{proposition}
\begin{proof}
   Let $S=\{(x_1,\ldots,x_s)\in(\Fq^*)^s\mid f(x_1,\ldots,x_s)=0\}\subset (\Fq^*)^s$ and, for each $i=1,\ldots,s$ and $r\mid q-1$, let 
\begin{equation}\label{eq01}
    \Lambda_i(r)=\{(x_1,\ldots,x_s)\in S\mid x_i\text{ is a }r\text{-th power}\}.
\end{equation}
It follows that
   $$\{(x_1,\ldots,x_s)\in S\mid (x_1,\ldots,x_s)\text{ is primitive}\}=S- \bigcup_{i=1}^s \bigcup_{\substack{
r \mid q-1 \\
r\text{ is prime}}}\Lambda_i(r).$$
Then we can use the Inclusion-Exclusion Principle to compute the value $$P_q(f(x_1,\ldots,x_s))=\big|\{(x_1,\ldots,x_s)\in S\mid (x_1,\ldots,x_s)\text{ is primitive}\}\big|.$$
We note that for each $s$-tuple $(x_1,\ldots,x_s)\in S$ such that each $x_i$ is a $r_i$-th power induces a rational point on the affine hypersurface $f(y_1^{r_1},\ldots,y_s^{r_s})=0$. In fact, we only need to set $y_i^{r_i}=x_i$ for each $i=1,\ldots,s$. On the other hand, each rational point $(y_1,\ldots,y_s)\in(\Fq^*)^s$ on the affine hypersurface $f(y_1^{r_1},\ldots,y_s^{r_s})=0$ is related to $r_1\cdots r_s$ $s$-tuples $(x_1,\ldots,x_s)\in S$, since the equation $y_i^{r_i}=x_i$ has exactly $r_i$ solutions over $\Fq$. Therefore,
\begin{equation}\label{eq02}
\big|\Lambda_1(r_1)\cap \dots\cap \Lambda_s(r_s)\big|=\frac{N^*(f(x_1^{r_1},\ldots,x_s^{r_s}))}{r_1\cdots r_s}.
\end{equation}
By using the Inclusion-Exclusion Principle in Equation~\eqref{eq01} and applying~\eqref{eq02}, we have that
\[
P_q(f(x_1,\ldots,x_s))= 
\displaystyle\sum_{\substack{
r_i \mid q-1 \\
i=1,\ldots,s}}
\frac{\mu(r_1)\cdots\mu(r_s)}{r_1\cdots r_s} N^*(f(x_1^{r_1},\ldots,x_s^{r_s})),
\]
which completes the proof.
\end{proof}


\section{Dwork-regular polynomials}\label{fermathy}

Combining the expressions given by Lemma~\ref{N*=NI} and Proposition~\ref{pf}, we obtain

$$P_q(f(\mathbf x))= 
\displaystyle\sum_{\mathbf r\mid (q-1)}
\sum_{I\subset\{1,\ldots,s\}}\frac{\mu(r_1)\cdots\mu(r_s)}{r_1\cdots r_s}(-1)^{s-|I|}
 N_I(f(\mathbf x^{\mathbf r})).$$
Since we aim to estimate $P_q(f(\mathbf x))$, we need a good estimate for each $N_I(f(\mathbf x^{\mathbf r}))$. Estimating the number of $\Fq$-rational points on hypersurfaces over finite fields is a compelling topic in algebraic geometry. For our purpose, it is essential to obtain a bound that depends linearly on each $r_i$, as we will see in the next computations. To ensure this, we will assume throughout this section that $f(x_1,\dots,x_s)\in\Fq(x_1,\dots,x_s)$ is a prime-to-$p$ degree $d$  Dwork-regular polynomial. Based on this hypothesis, we obtain the following estimate.

\begin{proposition} 
\label{NI=RI}
Assume that the polynomial $h\in\Fq[x_1,\ldots,x_s]$ of prime-to-$p$ degree $d$ is Dwork-regular. For each $s$-tuple $\mathbf r$ such that $\mathbf r\mid(q-1)$ and subset $I\subset\{1,\ldots,s\}$ with $I\neq \varnothing$, it follows that
 \[
 \left|N_I(h(\mathbf{x}^{\mathbf r}))-q^{|I|-1}\right|\le (\textstyle\prod_{i\in I}r_i) d^{|I|} q^{\frac{s}{2}} \le r_1 \cdots r_s d^{|I|} q^{\frac{s}{2}}.
 \]
\end{proposition}

\begin{proof}
Let $S_I=\{\mathbf x\in\Fq^s: x_i=0\text{ for each }i\in I^{\mathsf{c}}\}$. We note that Lemma~\ref{item123} implies that
$$N_I(h(\mathbf{x}^{\mathbf r}))=q^{-1}\sum_{\alpha \in\Fq}\sum_{\mathbf x\in S_I}\psi(\alpha h(\mathbf x^{\mathbf r})).$$
By letting $y_i=x_i^{r_i}$ and $\mathbf y=(y_1,\ldots,y_s)$, we use Lemma~\ref{dpower} in the above equation to obtain
$$\begin{aligned}N_I(h(\mathbf{x}^{\mathbf r}))&=q^{|I|-1}+q^{-1}\sum_{\alpha \in\Fq}\sum_{\mathbf y\in S_I}\psi(\alpha h(\mathbf y))\prod_{i\in I}\sum_{e_i=1}^{r_i} \chi_{r_i}^{e_i}(y_i)\\
&=q^{|I|-1}+q^{-1}\sum_{\alpha \in\Fq}\sum_{\substack{e_i=1,\ldots,r_i\\i\in I}}\sum_{\mathbf y\in S_I}\psi(\alpha h(\mathbf y)) \chi_{r_i}^{e_i}(y_i).\\
\end{aligned}$$

The result follows by applying the bound of Theorem~\ref{katz} in each term of the above sum.
\end{proof}

 Prior to actually computing $P_q(f(\mathbf x))$, it is convenient to compute 
$$N^*(f(\mathbf x^{\mathbf r}))= \displaystyle\sum_{I\subset\{1,\ldots,s\}} (-1)^{s-|I|} N_I(f(\mathbf x^{\mathbf r})).$$
If $I\neq\varnothing$, then Proposition \ref{NI=RI} states that
\[
 N_I(f(\mathbf x^{\mathbf r}))=q^{|I|-1}+R_{\mathbf r}(I),
\]
where $R_{\mathbf r}(I)$ is an integer that satisfies $\lvert R_{\mathbf r}(I) \rvert \leq r_1\cdots r_s d^{|I|} q^{\frac{|I|}{2}}$, and therefore 
\[
    N^*(f(\mathbf x^{\mathbf r}))=  \displaystyle\sum_{\substack{I\subset\{1,\ldots,s\}\\
    I\neq\varnothing}}(-1)^{s-|I|}\big(q^{|I|-1}+R_{\mathbf r}(I)\big)+(-1)^s N_{\varnothing}(f(\mathbf x))
\]
Furthermore, if we define
\[
R_{\mathbf r}(\varnothing):=N_{\varnothing}(f(\mathbf x))-q^{-1}=\begin{cases}
    1- q^{-1},&\text{ if }f(0,\ldots,0)=0\\
    -q^{-1},&\text{ otherise },\\
\end{cases}
\]
then $|R_{\mathbf r}(\varnothing)|<1\le r_1\cdots r_s.$ Considering this, it follows that
\[
    N^*(f(\mathbf x^{\mathbf r}))=  \displaystyle\sum_{I\subset\{1,\ldots,s\}}(-1)^{s-|I|}\big(q^{|I|-1}+R_{\mathbf r}(I)\big),
\]
where $\lvert R_{\mathbf r}(I) \rvert \leq r_1\cdots r_s d^{|I|} q^{\frac{|I|}{2}}$ for each $I\subset\{1,\ldots,s\}.$ We note that
\[
\displaystyle\sum_{I\subset\{1,\ldots,s\}}(-1)^{s-|I|}q^{|I|-1}=q^{-1}\displaystyle\sum_{k=0}^{s} (-1)^{s-k}\binom{s}{k}q^{k}
=\frac{(q-1)^{s}}{q},
\]
which implies that
\begin{equation}
\label{n*rfi}N^*(f(\mathbf x^{\mathbf r}))=\frac{(q-1)^{s}}{q}+R_{\mathbf r},
\end{equation}
where
\[
R_{\mathbf r}:=\displaystyle\sum_{I\subset\{1,\ldots,s\}}(-1)^{|I|}R_{\mathbf r}(I).
\]
We observe that
\begin{equation}
\label{rrbound}
|R_{\mathbf r}|\le \displaystyle\sum_{I\subset\{1,\ldots,s\}}|R_{\mathbf r}(I)|\le \sum_{I\subset\{1,\ldots,s\}} r_1\cdots r_s d^{|I|} q^{\frac{|I|}{2}}=r_1\cdots r_s\sum_{k=0}^s\binom{s}{k} d^k q^{\frac{k}{2}}=r_1\cdots r_s(d q^{\frac{1}{2}}+1)^s.
\end{equation}
According to Proposition~\ref{pf} and Equation~\eqref{n*rfi}, we have that
$$\begin{aligned}
P_q(f(\mathbf x))&= 
\displaystyle\sum_{\mathbf r\mid (q-1)}
\frac{\mu(r_1)\cdots\mu(r_s)}{r_1\cdots r_s}\left(\frac{(q-1)^{s}}{q}+R_{\mathbf r}\right)\\
&=\frac{(q-1)^{s}}{q}\displaystyle\sum_{\mathbf r\mid (q-1)}
\frac{\mu(r_1)\cdots\mu(r_s)}{r_1\cdots r_s}+\displaystyle\sum_{\mathbf r\mid (q-1)}
\frac{\mu(r_1)\cdots\mu(r_s)}{r_1\cdots r_s}R_{\mathbf r}.\\
&=\frac{1}{q}
\left(\displaystyle\sum_{r_1 \mid q-1}\mu(r_1)\frac{(q-1)}{r_1}\right) \cdots \left(\displaystyle\sum_{r_s \mid q-1}\mu(r_s)\frac{(q-1)}{r_s}\right)
+
\displaystyle\sum_{\mathbf r\mid (q-1)}
\frac{\mu(r_1)\cdots\mu(r_s)}{r_1\cdots r_s} R_{\mathbf r} \\
\end{aligned}
$$
By applying Lemma~\ref{mobiusEuler} to the above expression, we obtain that
\begin{equation}\label{finalformula}
P_q(f(\mathbf x))=\frac{\varphi(q-1)^s}{q}+
\displaystyle\sum_{\mathbf r\mid (q-1)}
\frac{\mu(r_1)\cdots\mu(r_s)}{r_1\cdots r_s} R_{\mathbf r}.\\
\end{equation}
Now, we can prove Theorem~\ref{teo1}.

\subsection{Proof of Theorem~\ref{teo1}}
Equation~~\eqref{finalformula} implies that 
\[
\big|P_q(f(\mathbf x))- \tfrac{\varphi(q-1)^s}{q}\big|\leq
\displaystyle\sum_{\mathbf r\mid (q-1)}
\frac{|\mu(r_1)|\cdots|\mu(r_s)|}{r_1\cdots r_s} |R_{\mathbf r}|=
\displaystyle\sum_{\substack{
\mathbf r\mid (q-1) \\
\mu(\mathbf r)\neq 0}}
\frac{|R_{\mathbf r}|}{r_1\cdots r_s},\
\]
where $\mu(\mathbf r)\neq 0$ means that $\mu(r_i)\neq 0$ for each $i=0,\ldots,s$.
Applying the bound for $|R_{\mathbf r}|$ obtained in Equation~\eqref{rrbound}, we have that
$$\begin{aligned}
\big|P_q(f(\mathbf x))- \tfrac{\varphi(q-1)^s}{q}\big|&\leq 
\displaystyle\sum_{\substack{
\mathbf r\mid (q-1) \\
\mu(\mathbf r)\neq 0}}(d q^{\frac{1}{2}}+1)^s\\
&=
(d q^{\frac{1}{2}}+1)^s\Bigg(\displaystyle\sum_{r_1\mid (q-1),\ \mu(r_1)\neq 0} 1\Bigg)^s\\
&=
(d q^{\frac{1}{2}}+1)^s W(q-1)^s.\\
\end{aligned}$$
This completes the proof of our result.\hfill\qed

\section{Fermat hypersurfaces}\label{fermathy}

A {\it Fermat hypersurface} is an affine hypersurface defined by the equation
\[
a_1x_1^{d_1}+a_2x_2^{d_2}+\cdots+a_sx_s^{d_s}=b
\]
where $a_1,\ldots,a_s, \, b  \in \mathbb{F}_q$, and $s\geq 2$.  

We will estimate the number of $\mathbb{F}_q$-primitive solutions of the Fermat hypersurface when  $d_i\mid (q-1)$, for $i=1,\ldots,s$. Note that $x_i$ is a primitive element of order $q-1$ in $\mathbb{F}_q^*$, if and only if, $x_i^{d_i}$ is a primitive element of order $\frac{q-1}{d_i}$ in $\mathbb{F}_q^*$. Thus, to determine the number of $\mathbb{F}_q$-primitive points on the Fermat hypersurface, when $d_i\mid (q-1)$, we calculate the number of primitive points on the hypersurface
\[
f:\quad a_1y_1+a_2y_2+\cdots+a_sy_s-b,
\]
where $y_i$ is an element of order $\frac{q-1}{d_i}$, for $i=1,\ldots,s$.

\[
\mathbb{I}_d(y)=
\begin{cases}
    1& \text{ if } y \text{ has order } \frac{q-1}{d}\text{ in }\Fq^*,\\
    0& \text{ otherwise. }
\end{cases}
\]
In Equation 2 of \cite{cohenreis}, the authors present the following reordered form of the expression for the indicator function $\mathbb{I}_d$, originally introduced by Carlitz (see \cite{carlitz}):

\begin{eqnarray}
    \label{for-ind}    
\mathbb{I}_d(y)=\frac{\varphi(\frac{q-1}{d})}{q-1}\sum_{r\mid (q-1)}\frac{\mu\big(\frac{r}{\gcd(r,d)}\big)}{\varphi\big(\frac{r}{\gcd(r,d)}\big)}
\sum_{ord(\lambda)=r}\lambda(y),
\end{eqnarray}
where $\lambda$ runs over the multiplicative characters of order $r$. Note that 

\begin{equation}\label{idezero}
\mathbb{I}_d(0)
= \frac{\varphi(\frac{q-1}{d})}{q-1}
\sum_{ r\mid (q-1)}\frac{\mu\big(\frac{r}{\gcd(r,d)}\big)}{\varphi\big(\frac{r}{\gcd(r,d)}\big)}
\sum_{ord(\lambda)=r}\lambda(0)= \frac{\varphi(\frac{q-1}{d})}{q-1},
\end{equation}

since $\lambda(0):=0$ for any nontrivial multiplicative character $\lambda$ and $\lambda(0):=1$ if $\lambda$ is trivial.

To verify some of our main results, we need the following two technical lemmas.

\begin{lemma}
\label{sum-mu-phi}\cite[Lemma 7]{cohenreis}
For any positive integers $r$, $d$, we have that 
\[
\sum_{r \mid q-1}\frac{\big|\mu\left(\frac{r}{(r,d)}\right)\big|}{\varphi(\frac{r}{(r,d)})}\varphi(r)
=(q-1,d)W\left(\big(d, \tfrac{q-1}{(q-1,d)}\big)\right).
\]
\end{lemma}

\begin{lemma}\label{lemad} If $y$ is a element of order $\frac{q-1}{d}$ in $\Fq^*$, then the number of $\Fq$-primitive elements $x$ such that $x^d=y$ is equal to
$$
\frac{\varphi(q-1)}{\varphi(\frac{q-1}{d})}.
$$
\end{lemma}
\begin{proof} Let $y\in\Fq$ be an element of order $\tfrac{q-1}{d}$. Then there exists a primitive element $\alpha$ of $\Fq^*$ such that $y=\alpha^d$. We want to compute the number of primitive elements $x\in\Fq$ such that $x^d=y=\alpha^d$. The solutions of the equation $x^d=\alpha^d$ in $\Fq$ are the elements $x=\alpha^{\frac{q-1}{d}i+1}$, where $i\in S=\left\{1,\ldots,d\right\}$.

We note that $\alpha^{\frac{q-1}{d}i+1}$ is primitive if and only if $\gcd\big(\frac{q-1}{d}i+1,q-1\big)=1.$ Therefore, we need the compute that cardinality of the set
$$P=\{i\in S:\gcd\big(\tfrac{q-1}{d}i+1,q-1\big)=1\}$$
Let  $t$ be the largest divisor of $d$ that is coprime to $\tfrac{q-1}{d}$ and, for a divisor $r$ of $t$, let
 $$\Lambda(r):=\big\{(i\in\{1,\ldots,d:r\text{ divides }\gcd\big(\tfrac{q-1}{d}i+1,q-1\big)\big\}=\big\{(i\in\{1,\ldots,d:r\text{ divides }\tfrac{q-1}{d}i+1\big\}.$$
 We note that $\gcd\big(\frac{q-1}{d}i+1,q-1\big)$ shares no common factors with $\frac{q-1}{d}$, which consequently implies that $\gcd\big(\frac{q-1}{d}i+1,q-1\big)$ must divide $t$. Therefore 
   $$P=S- \bigcup_{\substack{
r \mid t \\
r\text{ is prime}}}\Lambda_i(r).$$
Then we can use the Inclusion-Exclusion Principle to obtain
$$|P|=\sum_{r\mid t}\mu(r)|\lambda(r)|.$$
We note that if $r\mid t$, then $\gcd(r,\frac{q-1}{r})=1$, which implies that the congruence 
$$\tfrac{q-1}{d}i+1\equiv 0\pmod{r}$$
has exactly one solution $i\in\{1,\ldots,r\}$ so that $|\lambda(r)|=\tfrac{d}{r}$. Therefore,
$$|P|=\sum_{r\mid t}\mu(r)\tfrac{d}{r}=\tfrac{d}{t}\sum_{r\mid t}\mu(r)\tfrac{t}{r}=\tfrac{d}{t}\varphi(t)$$
by Lemma~\ref{mobiusEuler}. Since every prime factor of $\tfrac{t}{r}$ is also a factor of $\tfrac{q-1}{d}$, it follows that 
$$\varphi\big(\tfrac{q-1}{d}\cdot\tfrac{d}{t}\big)=\tfrac{d}{t}\varphi\big(\tfrac{q-1}{d}\big).$$
Applying the multiplicativity property of Euler's totient function, we then obtain
$$|P|=\tfrac{d}{t}\varphi(t)=\frac{\varphi\big(\tfrac{q-1}{d}\cdot\tfrac{d}{t}\big)\varphi(t)}{\varphi\big(\tfrac{q-1}{d}\big)}=\frac{\varphi\big(\tfrac{q-1}{d}\cdot\tfrac{d}{t}\cdot t\big)}{\varphi\big(\tfrac{q-1}{d}\big)}=\frac{\varphi\big(q-1\big)}{\varphi\big(\tfrac{q-1}{d}\big)},$$
which completes the proof.
\end{proof}

\subsection{Proof of Theorem~\ref{fermat-cota}} Let $\varepsilon_i=\frac{\varphi(q-1)}{\varphi(\frac{q-1}{d_i})}$ for each $i=1,\ldots,s$. It follows from Lemma~\ref{lemad} that

\[
P_q(f)=\sum_{\substack{a_1y_1+\cdots+a_sy_s=b \\ y_i\neq 0}}
\varepsilon_1 \mathbb{I}_{d_1}(y_1)\varepsilon_2 \mathbb{I}_{d_2}(y_2)\cdots\varepsilon_s \mathbb{I}_{d_s}(y_s).
\]
It follows from the Inclusion-Exclusion Principle and Equations~\eqref{for-ind} and~\eqref{idezero} that
\[\begin{aligned}
P_q(f)&=\sum_{I\subset \{1,\ldots,s\}}\sum_{\substack{a_1y_1+\cdots+a_sy_s=b \\ y_i= 0, \, \forall \, i \in I^c}} (-1)^{s-|I|}
\varepsilon_1 \mathbb{I}_{d_1}(y_1)\varepsilon_2 \mathbb{I}_{d_2}(y_2)\cdots\varepsilon_s \mathbb{I}_{d_s}(y_s)\\
&=\sum_{I\subset \{1,\ldots,s\}}\sum_{\substack{a_1y_1+\cdots+a_sy_s=b \\ y_i= 0, \, \forall \, i \in I^c}} (-1)^{s-|I|}\varepsilon^{s-|I|}\prod_{i\in I} \varepsilon_i \mathbb{I}_{d_i}(y_i),
\end{aligned}
\]

where $\varepsilon=\frac{\varphi(q-1)}{q-1}$. We note that if $I=\varnothing$, then
$$R_b:=\sum_{\substack{a_1y_1+\cdots+a_sy_s=b \\ y_i= 0, \\ \forall \, i \in \{1,\ldots,s\}}} (-1)^{s-|\varnothing|}\varepsilon^{s-|\varnothing|}\prod_{i\in \varnothing} \varepsilon_i \mathbb{I}_{d_i}(y_i)=\begin{cases}
    0,&\text{ if }b\neq 0\\
    (-1)^s\varepsilon^s,&\text{ if }b= 0.\\
\end{cases}$$
 
Therefore, by the definition of the  function $\mathbb{I}_d$, we have that

\[\small P_q(f)-R_b=
 \varepsilon^s\sum_{\substack{I\subset \{1,\ldots,s\}\\I\neq\varnothing}} \sum_{\substack{a_1y_1+\cdots+a_sy_s=b \\ y_i= 0, \, \forall \, i \in I^c}}\hspace{-0.5cm} (-1)^{s-|I|}\prod_{i\in I}\sum_{\substack{r_i\mid (q-1)\\\forall i\in I}}\frac{\mu\big(\frac{r_i}{(r_i,d_i)}\big)}{\varphi\big(\frac{r_i}{(r_i,d_i)}\big)}
\sum_{ord(\lambda_i)=r_i}\lambda_i(y_i)
\]
where $(r_i,d_i)=\gcd(r_i,d_i)$. Letting $N:=P_q(f)-R_b$, we have that

\[
N=\varepsilon^s\sum_{\substack{I\subset \{1,\ldots,s\}\\I\neq\varnothing}}  (-1)^{s-|I|}\sum_{\substack{r_i\mid (q-1)\\\forall i\in I}}\sum_{ord(\lambda_i)=r_i}\prod_{i\in I}\frac{\mu\big(\frac{r_i}{(r_i,d_i)}\big)}{\varphi\big(\frac{r_i}{(r_i,d_i)}\big)}
J_I(\lambda_{i_1},\ldots,\lambda_{i_{|I|}}),
\]
where $I=\{i_1,\ldots,i_{|I|}\}$ and
$$J_I(\lambda_{i_1},\ldots,\lambda_{i_{|I|}}):=\sum_{\sum_{i\in I}a_iy_i=b}\prod_{i\in I}\lambda_{i}(y_i)$$
is a Jacobi sum in the characters $\lambda_i$ with $i\in I$ in the cases where $|I|\ge 1$ and $J_\varnothing :=1$.
It follows from Proposition~\ref{item1519} that $J_I(\lambda_1,\ldots,\lambda_{|I|})=q^{|I|-1}$ if $|I|\ge 1$ and $\lambda_i$ is trivial for all $i\in I$. Furthermore, Proposition \ref{item1519} implies that $J_I(\lambda_1,\ldots,\lambda_{|I|})=0$ if some $\lambda_i$ is trivial. Separating these cases, we obtain
\[{\footnotesize
\begin{aligned}
    N&=\varepsilon^s\sum_{\substack{I\subset \{1,\ldots,s\}\\I\neq\varnothing}}  (-1)^{s-|I|}q^{|I|-1} +\varepsilon^s\sum_{\substack{I\subset \{1,\ldots,s\}\\I\neq\varnothing}}  (-1)^{s-|I|}\sum_{\substack{r_i\mid (q-1),r_i\neq 1\\\forall i\in I}}\sum_{ord(\lambda_i)=r_i}\prod_{i\in I}\frac{\mu\big(\frac{r_i}{(r_i,d_i)}\big)}{\varphi\big(\frac{r_i}{(r_i,d_i)}\big)}
J_I(\lambda_{i_1},\ldots,\lambda_{i_{|I|}})\\
&=\varepsilon^s\sum_{I\subset \{1,\ldots,s\}}  (-1)^{s-|I|}q^{|I|-1} +\tfrac{(-\varepsilon)^s}{q}+\varepsilon^s\sum_{\substack{I\subset \{1,\ldots,s\}\\I\neq\varnothing}}  (-1)^{s-|I|}\sum_{\substack{r_i\mid (q-1),r_i\neq 1\\\forall i\in I}}\sum_{ord(\lambda_i)=r_i}\prod_{i\in I}\frac{\mu\big(\frac{r_i}{(r_i,d_i)}\big)}{\varphi\big(\frac{r_i}{(r_i,d_i)}\big)}
J_I(\lambda_{i_1},\ldots,\lambda_{i_{|I|}})\\
&=\frac{\varphi(q-1)^s}{q}+\tfrac{(-\varepsilon)^s}{q}+\varepsilon^s\sum_{\substack{I\subset \{1,\ldots,s\}\\I\neq\varnothing}}  (-1)^{s-|I|}\sum_{\substack{r_i\mid (q-1),r_i\neq 1\\\forall i\in I}}\sum_{ord(\lambda_i)=r_i}\prod_{i\in I}\frac{\mu\big(\frac{r_i}{(r_i,d_i)}\big)}{\varphi\big(\frac{r_i}{(r_i,d_i)}\big)}
J_I(\lambda_{i_1},\ldots,\lambda_{i_{|I|}}).\\
\end{aligned}}
\]
It follows from Proposition \ref{item519} that

\[
|J_I(\lambda_{i_1},\ldots,\lambda_{i_{|I|}})|\leq q^{\frac{|I| -1}{2}}, 
\]
for every nonempty $I\subset \{1,\ldots,s\}$ provided $r_i>1$ for all $i\in I$. This implies that

\begin{equation}
\label{pfecota}
 \big|N-\tfrac{\varphi(q-1)^s}{q}\big|\le \frac{\varepsilon^s}{q}+\varepsilon^s\sum_{\substack{I\subset \{1,\ldots,s\}\\I\neq\varnothing}}  \left(\sum_{\substack{r_i\mid (q-1),r_i\neq 1\\\forall i\in I}}\prod_{i\in I}\frac{\big| \mu\big(\frac{r_i}{(r_i,d_i)}\big) \big|}{\varphi\big(\frac{r_i}{(r_i,d_i)}\big)}\varphi(r_i)q^{\frac{|I|-1}{2}}\right),
 \end{equation}
 since there exist $\varphi(r_i)$ multiplicative characters of $\Fq$ with order $r_i$.
 For ease of calculations, we denote the summation
 $$
\sum_{\substack{I\subset \{1,\ldots,s\}\\I\neq\varnothing}}  \left(\sum_{\substack{r_i\mid (q-1),r_i\neq 1\\\forall i\in I}}\prod_{i\in I}\frac{\big| \mu\big(\frac{r_i}{(r_i,d_i)}\big) \big|}{\varphi\big(\frac{r_i}{(r_i,d_i)}\big)}\varphi(r_i)q^{\frac{|I|-1}{2}}\right)
 $$
 by $\mathcal{S}$. Note that 
 $$
\mathcal{S}=\sum_{\substack{I\subset \{1,\ldots,s\}\\I\neq\varnothing}} q^{\frac{|I|-1}{2}} \prod_{i\in I}\left(\sum_{\substack{ r_{i}\mid (q-1)\\r_{i}\neq 1}}\frac{\big| \mu\big(\frac{r_{i}}{(r_{i},d_1)}\big) \big|}{\varphi\big(\frac{r_{i}}{(r_{i},d_1)}\big)}\varphi(r_{i})\right)
$$
Note that $d_i$ and $\frac{q-1}{d_i}$ divide $q-1$,  thus $(q-1,d_i)\,=\,d_i$ and $\left(q-1,\frac{q-1}{d_i} \right)\,=\,\frac{q-1}{d_i}$,  for $i=\,1,\ldots,\, s$. It follows from Lemma \ref{sum-mu-phi} and these observations that
$$\sum_{\substack{ r_{i}\mid (q-1)\\r_{i}\neq 1}}\frac{\big| \mu\big(\frac{r_{i}}{(r_{i},d_1)}\big) \big|}{\varphi\big(\frac{r_{i}}{(r_{i},d_1)}\big)}\varphi(r_{i})=d_iW\big(\tfrac{q-1}{d_i}\big)-1$$
for each $i=1,\ldots,s$. Therefore,
 $$\begin{aligned}
     \mathcal{S}&= \sum_{\substack{I\subset \{1,\ldots,s\}\\I\neq\varnothing}}q^{\frac{|I|-1}{2}} \prod_{i\in I}\left[d_iW\big(\tfrac{q-1}{d_i}\big)-1\right]\\
     &= \sum_{I\subset \{1,\ldots,s\}}q^{\frac{|I|-1}{2}} \prod_{i\in I}\left[d_iW\big(\tfrac{q-1}{d_i}\big)-1\right]-\tfrac{1}{\sqrt{q}}\\
     &=q^{-\frac{1}{2}}\left[1+\bigg(d_1W\big(\tfrac{q-1}{d_1}\big)-1\bigg)q^{\frac{1}{2}}\right]\dots\left[1+\bigg(d_sW\big(\tfrac{q-1}{d_s}\big)-1\bigg)q^{\frac{1}{2}}\right]-\tfrac{1}{\sqrt{q}}\\
 \end{aligned}
$$ 

Substituting $\mathcal{S}$ in Equation \ref{pfecota}, and using that $N:=P_q(f)-R_b$ and $\tfrac{1}{q}-\tfrac{1}{\sqrt{q}}<0$, we obtain that
$$\big|P_q(f)- \tfrac{\varphi(q-1)^s}{q}\big|\le \frac{\varepsilon^s}{\sqrt{q}}\left[1+\bigg(d_1W\big(\tfrac{q-1}{d_1}\big)-1\bigg)q^{\frac{1}{2}}\right]\dots\left[1+\bigg(d_sW\big(\tfrac{q-1}{d_s}\big)-1\bigg)q^{\frac{1}{2}}\right]+\delta_b,$$
where $\delta_b=\varepsilon^s$ if $b=0$ and $\delta_b=0$ if $b\neq 0$. This completes the proof of our result.\hfill\qed

\subsection{A sieving inequality} Our goal now is to provide a sieving version of Theorem~\ref{fermat-cota}. To achieve this, we utilize a prime sieving technique commonly applied to problems of this nature. For this purpose, we first require the concept of freenens, as introduced in \cite{cohenreis}.
\begin{definition}
    Let $C_d$ be the multiplicative subgroup of $\Fq^*$ of order $\tfrac{q-1}{d}$ and let $R$ be a divisor of $\tfrac{q-1}{d}$. An element $h\in\Fq^*$ is $(R,d)$-free if the following hold:
    \begin{enumerate}[label=(\roman*)]
        \item $h\in C_d$
        \item if $h=g^s$ with $g\in C_d$ and $s\mid R$, then $s=1$.
    \end{enumerate}
\end{definition}

Is is straightforward that $h\in\Fq$ has order $\tfrac{q-1}{d}$ if and only if $h$ is $(\frac{q-1}{d},d)$-free. Let $a_1,\dots,a_s,b\in\Fq^*$ and $s\ge 2$. We note that there exists a $\Fq$-primitive point on the Fermat hypersurface 
\begin{equation}\label{eqx}
a_1x_1^{d_1}+a_2x_2^{d_2}+\cdots+a_sx_s^{d_s}=b
\end{equation}
if and only if the equation
\begin{equation}\label{eqy}
    a_1y_1+a_2y_2+\cdots+a_sy_s=b
\end{equation}
has a solution $(y_1,\dots,y_s)\in\Fq^s$ where each $y_i$ has order $\tfrac{q-1}{d_i}$. 
For divisors $R_i$ of $\tfrac{q-1}{d_i}$, we let $N(R_1,\dots,R_s)$ be the number of $\Fq$-solutions $(y_1,\dots,y_s)\in\Fq^s$ of Equation~\eqref{eqy} where each $y_i$ is $(R_i,d_i)$-free.  Our goal is to establish a strong lower bound for $N\big(\tfrac{q-1}{d_1},\dots,\tfrac{q-1}{d_s}\big)$ through the prime sieving method. This allows us to ensure the existence of an $\Fq$-primitive point on the hypersurface given by Equation~\eqref{eqx}. To proceed, it is necessary to compute $N(R_1,\dots,R_s)$ for arbitrary divisors $R_i$ of $\tfrac{q-1}{d_i}$.

By following a similar approach as in the proof of Theorem~\ref{fermat-cota} and employing the formula for $(R,d)$-free elements provided in Proposition 13 of \cite{cohenreis}, it can be shown that 

$$N(R_1,\dots,R_s)\geq \frac{\varphi(R_1)}{R_1 d_1}\dots\frac{\varphi(R_s)}{R_s d_s}\left(\frac{(q-1)^s}{q} \, - \frac{1}{\sqrt{q}}\left[1+\big(d_1 W(R_1)-1\big)q^{\frac{1}{2}}\right]\cdots\left[1+\big(d_s W(R_s)-1\big)q^{\frac{1}{2}}\right]\right).
$$

In the following result, we present the prime sieve that will be used to relax the condition of Theorem~\ref{fermat-cota} for the existence of $\Fq$-primitive points. This is done by applying the Cohen-Huczynska prime sieve~\cite{cohen2003}.

\begin{proposition}\label{sieve}
    Let $\ell_i\mid \tfrac{q-1}{d_i}$ and $p_1^{(i)},\dots,p_{t_i}^{(i)}$ be all primes dividing $\tfrac{q-1}{d_i}$ but not $\ell_i$. Take $\delta=1-\sum_{i=1}^s\sum_{j=1}^{t_i}\tfrac{1}{p_j^{(i)}}$. Then
    $$N\big(\tfrac{q-1}{d_1},\dots,\tfrac{q-1}{d_s}\big)\ge \sum_{j=1}^{t_1} N(\ell_1 p_j^{(1)},\dots,\ell_s)+\dots+ \sum_{j=1}^{t_s} N(\ell_1,\dots,\ell_s p_j^{(s)})-\big(t_1+\dots+t_s-1\big)N(\ell_1,\dots,\ell_s).$$
\end{proposition}

We can employ the above result to obtain the following:

\begin{theorem}\label{sieveconsequence} Assume the notation and conditions presented in Proposition~\ref{sieve}. If $\delta>0$, then $N\big(\tfrac{q-1}{d_1},\dots,\tfrac{q-1}{d_s}\big)>0$ provided
    $$\frac{(q-1)^s}{\sqrt{q}}>\left(\frac{t_1+\dots+t_s-1}{\delta}+2\right)\left[1+\big(d_1 W(\ell_1)-1\big)q^{\frac{1}{2}}\right]\cdots\left[1+\big(d_s W(\ell_s)-1\big)q^{\frac{1}{2}}\right].$$
\end{theorem}
The proof of the results above has become standard and will therefore be omitted. For readers interested in the details, we refer to the proof of Theorem 3.4 in~\cite{cohen2018}.

\subsubsection{Proof of Theorem~\ref{conj}}
Consider the sphere
\[
f \, :\quad x^{2}+y^{2}+z^{2}=1
\]
over a finite field $\mathbb{F}_q$ with  characteristic different from 2. Let $N$ denote that number of $\Fq$-rational points on such sphere. By Theorem \ref{fermat-cota},  
\[
N\geq \tfrac{\varphi(q-1)^3}{q}-\left(\tfrac{\varphi(q-1)}{q-1}\right)^3\tfrac{1}{\sqrt{q}}\left[1+\bigg(2 W\big(\tfrac{q-1}{2}\big)-1\bigg)q^{\frac{1}{2}}\right]^3.
\]

By using Lemma~\ref{boundW} to bound $W\big(\tfrac{q-1}{2}\big)$, the above inequality yields that a sufficient condition for $N>0$ is $q>5.275\times 10^9$. By choosing $\ell_1=\ell_2=\ell_3=:\ell$ and the same primes $p_1,\dots,p_t$ for each $\ell_i$ in Proposition~~\ref{sieve} and Theorem~\ref{sieveconsequence}, we have that $N>0$ provided
\begin{equation}\label{eqfinal}
\frac{(q-1)^3}{\sqrt{q}}>\left(\frac{3t-1}{\delta}+2\right)\left[1+\big(2 W(\ell)-1\big)q^{\frac{1}{2}}\right]^3.
\end{equation}
where $\delta=1-\sum_{j=1}^{t}\tfrac{3}{p_j}$. We observe that if $\tfrac{q-1}{2}$ is divisible by 10 or more distinct primes, then $q>12\times 10^9$, which implies that the case $W\big(\tfrac{q-1}{2}\big)\ge 2^{10}$ is settled. We may choose as sieving primes the largest number of prime divisors of $\tfrac{q-1}{2}$ for which $\delta>0$. We accomplish this by choosing the largest prime divisors of $\tfrac{q-1}{2}$. In the worst-case scenario, the prime factors of $\tfrac{q-1}{2}$ will consist of the smallest prime numbers. In the following table, we present the chosen sieving primes and the interval within which Inequality~\eqref{eqfinal} is satisfied.


\begin{center}
\begin{tabular}{|c|c|c|c|c|}
\hline
Maximum $W\big(\tfrac{q-1}{2}\big)$ & $p_1,\ldots,p_t$ & $\delta$ & $W(\ell)$ &  Interval of vality \\ \hline
$2^9$ & $13,17,19,23$ & $0.304\ldots$ & $2^5$ &  $9.536\times 10^6<q\le5.275\times 10^9$\\ \hline
$2^7$ & $11,13,17$ & $0.320\ldots$ & $2^4$ &  $804377<q\le 9.536\times 10^6$\\ \hline
$2^6$ & 11,13 & $0.298\ldots$ & $2^3$ & $300067<q\le 804377$ \\ \hline
\end{tabular}
\end{center}
For the first two rows of the table, we utilized the inequalities $29\#>6\times 10^9$ and $19\#>9.6\times 10^6$ to select the worst possible sieving primes based on these primorials. In the third row, the largest value of $W(\tfrac{q-1}{2})$ was computed using a routine implemented in the SageMath mathematics software system (see code below).
\begin{lstlisting}
max = 804377 # Set the maximum limit for calculations
P = Primes() 
max_prime_factors=1
for p in P:

    # Stop if the prime exceeds the limit
    if p >= max:
        break  
    n = 1
    
    # Does not compute the case p=2
    if p==2:
        p=3
        
    while (q := p**n) < max:
        derived_value = (q - 1) // 2  
        
        # Count prime factors of the derived value
        num_prime_factors = len(prime_factors(derived_value))
            
        # Update maximum prime factors and add number to intervals
        if num_prime_factors > max_prime_factors:
            max_prime_factors += 1
            
        n += 1  # Increment exponent for prime power generation
print(f"Maximum number of prime factors:",max_prime_factors)
\end{lstlisting}
For prime powers $q \leq 18602$, we explicitly verified the existence of $\mathbb{F}_q$-primitive points on the sphere using a routine implemented in SageMath (see code below), excluding the cases $q\neq 3,5,9,13,25.$ Below is an optimized parallelized Python implementation that efficiently checks for primitive points in finite fields.

\begin{lstlisting}
import multiprocessing
from sympy import primerange
from sage.all import GF

max_val = 300067
manager = multiprocessing.Manager()
failed_q_values = manager.list()  # Using a shared list instead of set()

def primitive_point_exists(S):
    """Check if there exists a primitive point on the sphere"""
    return any(x**2 + y**2 + z**2 == 1 for x in S for y in S for z in S)

def process_prime(p):
    """Compute primitive points for a given prime power"""
    n = 1
    while (q := p**n) < max_val:
        F = GF(q)
        S = {a for a in F if a != 0 and a.multiplicative_order() == q-1}

        if not primitive_point_exists(S):
            failed_q_values.append(q)  

        n += 1  # Increment exponent for prime power generation

if __name__ == "__main__":

    primes = list(primerange(2, max_val))
    num_processes = min(multiprocessing.cpu_count(), len(primes))

    with multiprocessing.Pool(processes=num_processes) as pool:
        pool.map(process_prime, primes)

    print("Values of q for which the conjecture fails:", list(failed_q_values))  # Convert list to normal Python list
\end{lstlisting}

This completes the proof of the conjecture.\hfill\qed


\section{Fermat prime characteristic cases}

In this section, we will calculate the number of $\mathbb{F}_q$-primitive points on the hyperplane
		\begin{equation}\label{linear}
		a_1x_1+a_2x_2+\cdots+a_sx_s=b
		\end{equation}
		where $a_1,\ldots,a_s, \, b  \in \mathbb{F}_q$, $s\geq 2$, and $q>3$ is a Fermat prime, that is, $q = 2^{2^l} + 1$ for some positive integer $l>0$. For $f(x_1,\ldots,x_n)=a_1x_1+a_2x_2+\cdots+a_sx_s-b$, Theorem \ref{pf} reads
		
		\[
		P_q(f(\mathbf{x}))= 
		\displaystyle\sum_{\substack{
				r_1,\ldots,r_s \\
				r_i \in\{1,2\}}}
		\frac{\mu(r_1)\cdots\mu(r_s)}{r_1\cdots r_s} N^*(f(x_1^{r_1},\ldots,x_s^{r_s}),
		\]
    since the only squarefree divisors of  $q-1 = 2^{2^l}$ are 1 and 2. Furthermore,

			\[
			N^*(f(\mathbf{x}^\mathbf r))= 
			\displaystyle\sum_{I\subset\{1,\ldots,s\}} (-1)^{s-|I|}
			N_I(f(\mathbf{x}^\mathbf r))),
			\]
	  by Lemma \ref{N*=NI}. Thus, we obtain  
	  
\begin{eqnarray}
	\label{eq_pri_fermat}
P_q(f(\textbf{x}))= 
\displaystyle\sum_{\substack{
		r_1,\ldots,r_s \\
		r_i \in\{1,2\}}}
\dfrac{\mu(r_1)\cdots\mu(r_s)}{r_1\cdots r_s}
\displaystyle\sum_{I\subset\{1,\ldots,s\}} (-1)^{s-|I|}
N_I(\mathbf{x}^{\mathbf{r}})).
\end{eqnarray}

In light of Equation \ref{eq_pri_fermat}, in order to determine $P_q(f(\textbf{x}))$, we still need to calculate \newline $N_I(f(\mathbf{x}^{\mathbf{r}})$ for every $I \subset \{1, \ldots, s\}$.	To accomplish this, we will employ two well-known results from the literature, as listed below:

\begin{lemma}[Theorem 6.26, \cite{lidl1997finite}]
\label{qfeven}
Let $f$ be a nondegenerate quadratic form over $\mathbb{F}_q$, $q$ odd, in an even number $n$ of indeterminates. Then for $b \in \mathbb{F}_q $ the number of solutions of the equation $f(x_1,\ldots,x_s) = b$  in $\mathbb{F}_q^s$ is
\[
q^{s-1}+q^{\frac{s-1}{2}}\nu(b)\chi_2\left( (-1)^{\frac{s}{2}}\Delta \right),
\]
where $\chi_2$ is the quadratic character of $\mathbb{F}_q^*$, $\Delta=det(f)$, $\nu(0)=q-1$ and $\nu(b)=-1$ for $b\in\Fq^*$.
\end{lemma}

\begin{lemma}[Theorem 6.27, \cite{lidl1997finite}]
\label{qfodd}
Let $f$ be a nondegenerate quadratic form over $\mathbb{F}_q$, $q$ odd, in an odd number $s$ of indeterminates. Then for $b \in \mathbb{F}_q $ the number of solutions of equation $f(x_1,\ldots,x_s) = b$  in $\mathbb{F}_q^n$ is
\[
q^{s-1}+q^{\frac{s-1}{2}}\chi_2\left( (-1)^{\frac{s-1}{2}}b\Delta \right),
\]
where $\chi_2$ is the quadratic character of $\mathbb{F}_q$ and $\Delta=det(f)$.
\end{lemma}

\begin{theorem}
\label{ri=1}
Let $f(x_1,\ldots,x_s)=a_1x_1+a_2x_2+\cdots+a_sx_s-b$ a linear form over $\mathbb{F}_q$, where $q$ is a Fermat prime. If there exists $i \in I$ such that $r_i=1$, then 
\[
N_I(f(x_1^{r_1},\ldots,x_s^{r_s}))=q^{|I|-1},
\]
for $I \subset \{1, \ldots, s\}$ and $r_i\in \{1, 2 \}$.
\end{theorem}		
\begin{proof}
Assume, without loss of generality, that $1 \in I$ and $r_1=1$. Now observe that we can assign arbitrary elements of $\mathbb{F}_q$ to $x_2,\ldots,x_s$ and the value of $x_1$ is uniquely determined. This result follows directly from this observation.
\end{proof}

\begin{theorem}\label{quadratic}
	Let $f(x_1,\ldots,x_s)=a_1x_1+a_2x_2+\cdots+a_sx_s-b$ a linear form over $\mathbb{F}_q$, where $q$ is a Fermat prime. If all $i \in I$ we have  $r_i=2$, then 
	
	\begin{enumerate}
		\item 
		\[
		N_I(f(x_1^{r_1},\ldots,x_s^{r_s}))=q^{|I|-1}+q^{\frac{|I|-2}{2}}\nu(b)\chi_2\left(\prod_{i \in I}a_i\right),
		\] if $|I|$	 is even. Or
		
		\item 
		\[
		N_I(f(x_1^{r_1},\ldots,x_s^{r_s}))=q^{|I|-1}+q^{\frac{|I|-1}{2}}\chi_2\left(b\prod_{i \in I}a_i \right),
		\] if $|I|$ is odd.
	\end{enumerate}
\end{theorem}
		
\begin{proof}
It follows from Lemma~\ref{quachar} that $\chi_2(-1)=(-1)^{2^{2^l-1}}=1$, since $l>0$. Then the result follows directly from Lemmas \ref{qfeven} and \ref{qfodd}.
\end{proof}

	

\subsection{Proof of Theorem~\ref{numexato}}
    By replacing the values of $N_I(\mathbf x^\mathbf r)$ given by Theorems~\ref{ri=1} and~\ref{quadratic} into Equation~\ref{eq_pri_fermat} and applying Lemma~\ref{mobiusEuler}, we obtain that
    $$P_q(f(\textbf{x}))= \frac{\varphi(q-1)^s}{q}+
\displaystyle\sum_{\substack{
		I\subset\{1,\ldots,s\}\\
		|I|\text{ is even}}} \big(\tfrac{-1}{2}\big)^{s}q^{\frac{|I|-2}{2}}\nu(b)
\prod_{j\in I}\chi_2(a_j)+
\displaystyle\sum_{\substack{
		I\subset\{1,\ldots,s\}\\
		|I|\text{ is odd}}} \big(\tfrac{-1}{2}\big)^{s}q^{\frac{|I|-1}{2}}\chi_2(b)
\prod_{j\in I}\chi_2(a_j).$$
Then the result follows from the equalities
$$\displaystyle\sum_{\substack{
		I\subset\{1,\ldots,s\}\\
		|I|\text{ is even}}}q^{\frac{|I|}{2}}
\prod_{j\in I}\chi_2(a_j)=\frac{(-1)^s \prod_{i=1}^s\big(\sqrt{q}\chi_2(a_i)+1\big)+\prod_{i=1}^s\big(\sqrt{q}\chi_2(a_i)-1\big)}{2}$$
and
$$\displaystyle\sum_{\substack{
		I\subset\{1,\ldots,s\}\\
		|I|\text{ is odd}}}q^{\frac{|I|}{2}}
\prod_{j\in I}\chi_2(a_j)=\frac{(-1)^s\prod_{i=1}^s\big(\sqrt{q}\chi_2(a_i)+1\big)-\prod_{i=1}^s\big(\sqrt{q}\chi_2(a_i)-1\big)}{2}.$$
\hfill\qed


\end{document}